\documentclass[11pt,bezier]{article}
\usepackage{amsmath,graphicx,amssymb,amsfonts}

\newtheorem{prethm}{{\bf Theorem}}

\newenvironment{thm}{\begin{prethm}{\hspace{-0.5
               em}{\bf}}}{\end{prethm}}
\newtheorem{prepro}{{\bf Theorem}}

\newtheorem{preprop}{{\bf Proposition}}

\newenvironment{prop}{\begin{preprop}{\hspace{-0.5
               em}{\bf}}}{\end{preprop}}
\newtheorem{preex}{{\bf Example}}

\newtheorem{precor}{{\bf Corollary}}

\newtheorem{preconj}{{\bf Conjecture}}

\newtheorem{preremark}{{\bf Remark}}

\newtheorem{prefact}{{\bf Fact}}

\newenvironment{fact}{\begin{prefact}{\hspace{-0.5
               em}{\bf}}}{\end{prefact}}

\newtheorem{prelem}{{\bf Lemma}}

\newenvironment{lem}{\begin{prelem}{\hspace{-0.5
               em}{\bf}}}{\end{prelem}}
\newtheorem{preproof}{{\em Proof.}}

\newenvironment{proof}[1]{\begin{preproof}{\rm
               #1}\hfill{$\Box$}}{\end{preproof}}

\title{\bf Coloring Problem of Signed Interval Graphs}
\author{\bf  F. Ramezani \and
  \\ {\it Faculty of Mathematics,
 K. N. Toosi University of Technology,}\\
 {\it  P. O. Box $16315$--$1618$,  Tehran, Iran}
 \\[.3cm]
 {\it E-mail:} \ {\tt ramezani@kntu.ac.ir}
 \date{}}

\begin{document}
\maketitle

\begin{abstract}
\noindent A signed graph $(G,\sigma)$ is a graph
together with an assignment of signs $\{+,-\}$ to its edges where
$\sigma$ is the subset of its negative edges. There are a few variants of coloring and clique problems of
signed graphs, which have been studied. An initial version known as vertex coloring of signed graphs is defined by Zaslavsky in $1982$. Recently Naserasr et al ( in R. Naserasr, E. Rollova and E. Sopena, Homomorphisms of signed graphs, J.
Graph Theory, 79 (2015) 178--212) have defined signed chromatic and signed clique numbers of signed graphs. In this paper we consider the latter mentioned problems for signed $l_0$-interval graphs, that is signed graphs on an interval graph with at most $l_0$ maximal cliques. We prove that the coloring problem of signed
$l_0$-interval graphs is NP-complete whereas their ordinary coloring
problem (the coloring problem of interval graphs) is in P. Moreover we prove that the signed clique problem of a
signed $l_0$-interval graph can be solved in polynomial time. We also consider the
complexity of further related problems.
\end{abstract}

\section{Introduction}
We consider simple graphs $G=(V,E)$, i.e graphs without loops
and multiple edges. A graph $G$ together with a function
$s:E\longrightarrow \{+,-\}$ on the edge set of $G$ is called a
\textit{signed graph}. If $\sigma$ is the set of edges whose image under $s$ is $"-"$, then we denote the signed graph
by $\Sigma=(G,\sigma)$. The graph $G$ is called the ground
of $\Sigma$ and the set $\sigma$ is called the
signature of it. For any edge $e$ of $\Sigma$, we call it a
positive or negative edge if $s(e)$ has positive or negative sign
respectively. By the edge and vertex set of $\Sigma$ we mean
those of the ground graph that are $V,E$ respectively. For a
signed graph $\Sigma=(G,\sigma)$ by the \textit{positive}
(\textit{negative }) \textit{subgraph} we mean the spanning subgraph of
$G$ where the edge set is the set of positive (negative) edges of
$\Sigma$ and is denoted by $\Sigma^+$ ($\Sigma^-$).

For a cycle $C$ of $G$, the signature of $C$ in $(G,\sigma)$ is
the product of signs of its edges. A cycle is called
\textit{balanced} if it has positive signature, otherwise we call
it \textit{unbalanced}. An unbalanced cycle of length $n$ will be denoted by $UC_n$.
A signed graph all of whose cycles are
balanced is called \textit{balanced} otherwise we call it \textit{unbalanced}. By
\textit{resigning} at a vertex $v$ of $\Sigma$ we mean to change
signs of all the edges incident with $v$. Two signed graphs
$(G,\sigma)$ and $(G,\sigma')$ are called \textit{switching
equivalent} if one is obtained from the other by a sequence of
resignings. A well-known theorem of Zaslavsky states that two
signed graphs $(G,\sigma)$ and $(G,\sigma')$ are switching
equivalent if and only if they have the same set of unbalanced
cycles, see \cite{ZAS}.
A signed graph $(H,\sigma')$ is called a \textit{signed subgraph} of
$(G,\sigma)$ if $H$ is a subgraph of $G$ and
$\sigma'=\sigma\cap E(H)$. For a subset $S$ of $V$, by $\langle
S\rangle_\Sigma$ we denote the subgraph of $\Sigma$ whose ground
is the induced subgraph of $G$ on $S$. If there is no doubt
about $\Sigma$ then we simply write $\langle S\rangle$. For two
subsets $X,Y$ of $V$, by $\langle X,Y\rangle_{\Sigma}$ we denote
the subgraph of $\Sigma$ whose ground is the bipartite
subgraph of $G$ on the partition $X, Y$ with all possible edges
between $X$ and $Y$.

There are few variants of definitions for coloring problem of signed graphs. An initial definition due to Zaslavsky, is stated in \cite{ZAS1}. There a proper vertex colouring of $\Sigma=(G,\sigma)$ is defined to be a mapping $\phi: V(G)\longrightarrow \mathbb{Z}$ such that for each edge $e=uv$ of $G$ the colour $\phi(u)$ is distinct from the
 colour $s(e)\phi(v)$, where $s(e)$ is sign of $e$.
In other words, the colours of nodes of a positive edge must not be the same while
those of a negative edge must not be opposite. A $k$-colouring of a signed graph $\Sigma=(G,\sigma)$, is defined to be a colouring $V(G)\longrightarrow \{-k,-(k-1),\ldots,0,\ldots,k-1,k\}$. Therefore the chromatic number $\gamma(\Sigma)$ of $\Sigma$, according to Zaslavsky \cite{ZAS0}, is the smallest nonnegative integer $k$ for which $\Sigma$ admits a proper $k$-coloring. With respect to Zaslavsky's proper coloring definition, Macajova et al in \cite{MRS} have proposed a new definition of chromatic number of signed graphs to be the whole number of signed colours used in an optimal colouring. This modifies Zaslavsky's definition to a natural extension of ordinary graphs chromatic number. See \cite{MRS} for more details.

Naserasr et al in \cite{NAS1} has introduced the idea of homomorphism of signed graphs. Their definition of homomorphism of signed graphs follows. Let $\Sigma=(G,\sigma)$ and $\Pi=(H,\pi)$ be given. A function $\Phi:V(\Sigma)\rightarrow V(\Pi)$ is called a homomorphism of $\Sigma$ to $\Pi$ if there is a switching equivalent pair of $\Sigma$ say $\Sigma'=(G,\sigma')$ such that for any pair $u,v$ of vertices of $\Sigma'$ if there is a positive (negative) edge between $u,v$ in $\Sigma'$ then there is a positive (negative) edge between $\Phi(u),\Phi(v)$ in $\Pi$. Based on this definition the \textit{signed chromatic number} of a signed graph $\Sigma$, denoted by $\chi_s(\Sigma)$, is the smallest order of a homomorphic image of it. Naserasr et al in section $3.4$ of their paper proposed an equivalent definition of signed chromatic number based on the following proper coloring notion. An assignment of $k$ colors to the vertices of $(G,\sigma)$ is called a \textit{proper coloring} if it satisfies the following properties.
\begin{itemize}
  \item Colors of adjacent vertices must be different,
  \item Edges
with different signs must not have the same set of colors on their end nodes.
\end{itemize}

Based on the previous definition, for a signed graph $(G,\sigma)$ the smallest number $k$ for which
there is a switching equivalent pair of $(G,\sigma)$
say $(G,\sigma_1)$, which admits a proper coloring with $k$ colors, is its \textit{signed
chromatic number}(denoted by $\chi_s(G,\sigma)$ or $\chi_s(\Sigma)$). A signed graph $(H,\sigma_0)$ on $n$ vertices
is called a \textit{signed clique} or \textit{S-clique} for short if its signed chromatic number equals $n$. The signed graph $UC_4$ is an example of a S-clique. Given a signed graph $(G,\sigma)$ its \textit{ S-clique number} is the maximum size of a signed subgraph of it which is a S-clique. The S-clique number of $(G,\sigma)$ will be denoted by $w_s(G,\sigma)$. Both the coloring and homomorphism problems of signed graphs are studied widely. In \cite{BFHN}  it is proved that deciding whether there exists a signed graph homomorphism to any fixed signed graph, except some special cases, is NP-complete. It is also discussed that the coloring problem in the sense of Zaslavsky can be reformulated as signed graph homomorphism of some specified signed graphs, this implies that Zaslavsky coloring problem is NP-complete. In \cite{NAS1} it is also proved that the coloring problem defined by Naserasr et al is NP-complete.

For a vertex $v$ of a graph $G$, by $N_G(v)$ (or simply $N(v)$) we denote the set of
neighbors of $v$ in $G$ and by $N_G[v]$ (or simply $N[v]$) we
denote the closed neighborhood of $v$ in $G$, that is,
$N(v)\cup\{v\}$.

An interval graph $I$ is the intersection graph of a set of real intervals. There exist some
handy characterizations of them which are useful in this paper. For
instance the following lemma from \cite{FI} provides a useful
classification of these graphs.

\begin{lem} \label{or}\cite{FI} The graph $I$ is an interval graph
if and only if the maximal cliques of it can be ordered $M_1,
M_2, ..., M_l$ such that for any $i,j,s$, where $i \leqslant j \leqslant
s$, it has the property $M_i \cap M_s\subseteq M_j$.

\end{lem}

Let $l_0\geq 2$ be a fixed positive integer. In this paper we consider signed $l_0$-interval graphs, i.e
signed graphs with an interval graph as their ground, in which there exist less than $l_0+1$ maximal cliques in the ground.
The coloring problem INTSCOL of signed $l_0$-interval graphs is based on the definition by Naserasr et al as already mentioned and is defined in the following.

\noindent \textbf{INTSCOL}\\
Input: A signed $l_0$-interval graph $\mathcal{I}$ and a positive integer $k$.\\
Question: Is $\chi_s(\mathcal{I})\leq k$?

The S-clique problem of signed $l_0$-interval graphs called INTSCLIQ is the following.

\noindent \textbf{INTSCLIQ}\\
Input: A signed $l_0$-interval graph $\mathcal{I}$ and a positive integer $k$.\\
Question: Is $w_s(\mathcal{I})\geqslant k?$

We prove the following main theorems beyond other related results.

\begin{thm}  The INTSCLIQ problem is in P.
\end{thm}

\begin{thm} The INTSCOL problem is NP-complete.
\end{thm}


\section{Preliminaries}
We say that the vertex set of a graph $G$ has a \textit{perfect
elimination ordering} (abbreviated P.E.O) if there is an ordering
$v_1,v_2,\ldots,v_n$ of vertices of $G$, such that the subgraph
of $G$ induced on the vertex set $\{v_i,\ldots,v_n\}\cap N_G[v_i]$
forms a clique. It is well known that any interval graph $I$
admits a P.E.O, see for instance \cite{SHI}. Using this ordering,
the vertices of an interval graph $I$ can be colored with $\omega$
(clique number of $I$) colors by a greedy algorithm efficiently.
Hence the coloring problem of interval graphs is in P, see
\cite{LO} for more details. A perfect elimination ordering of vertices of an interval graph
can be obtained by the property already mentioned in Lemma \ref{or}.

In the following lemma from \cite{NAS1} an equivalent statement for a signed graph to be
a S-clique is established.
\begin{lem} \label{scli}  A signed graph is a S-clique if
and only if for each pair $u$ and $v$ of vertices either $uv$ is
an edge or $u$ and $v$ are vertices of an unbalanced cycle of
length $4$.
\end{lem}

For more convenience, we define a neighborhood vector of a vertex in a signed graph as
follows. Let $\Sigma$ be a signed graph on the ordered vertex set
$V=\{v_1,v_2,\ldots,v_n\}$ and $1\leqslant i_1<i_2<\cdots<i_s\leqslant n$
be an increasing sequence of integers. For
$S=\{v_{i_1},v_{i_2},\ldots,v_{i_s}\}$ and a vertex $v$ of
$\Sigma$ we define the $S-$neighborhood vector of $v$, denoted by
$\overrightarrow{N}_S(v)$, as below.
$$\overrightarrow{N}_S(v)_j=
  \begin{cases}
    1 & \text{if there is a positive edge between $v$ and $v_{i_j}$}, \\
    -1 & \text{if there is a negative edge between $v$ and $v_{i_j}$}, \\
    0&  \text{otherwise}.
  \end{cases}$$

The $V-$neighborhood vector of a vertex $v$ is simply denoted by $\overrightarrow{N}(v)$.
By the definition just mentioned we can restate Lemma \ref{scli}. A signed graph on an ordered set $V$ is a S-clique if
and only if for each pair $u$ and $v$ of vertices either $u$ is adjacent to $v$ or for $S=N(u)\cap N(v)$, $\overrightarrow{N}_S(u)\neq \pm\overrightarrow{N}_S(v)$.

An induced complete bipartite subgraph of a graph is called a \textit{biclique}. Here by \textit{size} of a biclique we mean its number of vertices. Note that a biclique may consist only of an independent set of vertices. The maximum biclique problem of a graph abbreviated to MBC, is defined below.

\noindent \textbf{MBC}\\
Input: A graph $G$ and a positive integer $k$.\\
Question: Does $G$ contain a biclique of size more than or equal
to $k$?

A set $L$ of edges of a graph $G$ which, when removed increases the number of components of the graph is an edge cut of $G$. The maximum edge cut problem is defined below.

\noindent \textbf{MEC}\\
Input: A graph $G$ and a positive integer $k$.\\
Question: Does $G$ contain an edge cut with more than or equal to $k$ edges?

A \textit{clique} in a graph $G$ is a subgraph of it all of whose vertices are adjacent. The maximum size of a clique in $G$ is called the clique number of it and is denoted by $\omega(G)$.  The definition of maximum clique problem in graphs, abbreviated to MC follows.

\noindent \textbf{MC}\\
Input: A graph $G$ and a positive integer $k$.\\
Question: Is $\omega(G)>k$?

A set $U$ of vertices of a graph is called \textit{independent} if there is no edge between them. The maximum size of an independent set in a graph is called the \textit{independence number} of graph. In what follows, we preview the complexity of the above problems from the literature which will be useful later.

\begin{fact} \label{fact-1}{\rm  \cite{SCH}} The maximum independence number of a bipartite graph
can be found in polynomial time.
\end{fact}

\begin{fact} \label{fact-2}{\rm
\cite{PEE}} The MC problem in graphs is NP-complete.
\end{fact}

\begin{fact} \label{fact-3}{\rm   \cite{PEE}}  The MBC problem in graphs is NP-complete.
\end{fact}

\textbf{Remarks.} We need to clarify validity of Fact $2$ and $3$ by using the method mentioned in \cite{PEE}. A graph property is called \textit{hereditary} on induced subgraphs if for any graph which satisfies the property, any vertex induced subgraphs of it also satisfy the property. A property is called nontrivial if it is true for infinitely many graphs and also false for infinitely many graphs. For a fixed graph and property $\pi$, the \textit{node-deletion problem} is finding the minimum number of vertices which must be deleted from the graph so that the result satisfies $\pi$. Lewis and Yannakakis \cite{PEE} prove that the node-deletion problem for any hereditary (on induced subgraphs), and nontrivial graph property is NP-complete. Note that this applies to the  MBC (and MC) problem since the property "being a biclique" ("being a clique") is hereditary (on induced subgraphs) and nontrivial, and maximizing the size of the biclique (clique) is the same as minimizing the number of vertices to be deleted so that the obtained subgraph satisfies the property.

\begin{fact} \label{fact-4}{\rm  \cite{SCH}} The MEC problem in a
graph is NP-complete.
\end{fact}

We will make advantage of the following theorem by Harary, \cite{HA}.

\begin{thm} \label{har}
Two signed graphs $(G,\sigma)$ and $(G,\sigma')$ are switching
equivalent if and only if the symmetric difference of $\sigma$ and $\sigma'$ is an edge cut of $G$.
\end{thm}
As a consequence of the above theorem note that a balanced signed graph $(G,\sigma)$ is
switching equivalent to $(G,\emptyset)$ in which $\sigma$ is an edge cut of $G$.

\section{Main Results}
In this section we present our main results on the complexity of
the INTSCOL and INTSCLIQ problems of signed interval graphs.
There will be also some other related results. It is known that
the coloring and clique problems of the interval graphs both
belong to P. In this section these problems for signed
interval graphs are studied.

\begin{thm}  The INTSCLIQ problem is in P.
\end{thm}
\begin{proof} {We prove the assertion by induction on the number of maximal cliques of $I$. The assertion already holds for the interval graphs with only one maximal cliques. Now assume that
$I$ has $l$ ($l_0\geq l\geq 2$) maximal cliques. Let $M_1, M_2, ..., M_l$ be
the maximal cliques of $I$ with the ordering mentioned in Lemma
\ref{or}. For the signed graph $\mathcal{I}=(I,\sigma)$, we
define a graph $\mathcal{I}^*$ as following. The vertex set
of $\mathcal{I}^*$ is $V(I)$ in which there is an edge between
two vertices if they are adjacent in $\mathcal{I}$ or they belong
to an unbalanced $4$-cycle of $\mathcal{I}$. By Lemma \ref{scli},
a set $A$ of vertices is a maximum signed clique in $\mathcal{I}$
if and only if it is a maximum clique of $\mathcal{I}^*$. Hence
it suffices to consider the maximum clique problem in
$\mathcal{I}^*$, or equivalently the maximum independent set
problem in ${\mathcal{I}^*}^c$, its complement. Note that two
vertices $u,v$ are adjacent in ${\mathcal{I}^*}^c$ only if they
are non-adjacent in $\mathcal{I}^*$, that is, if $u,v\in
V(\mathcal{I}^*)$ are adjacent then in the signed graph
$\mathcal{I}$ the equality $\overrightarrow{N}_{L}(u)=\pm
\overrightarrow{N}_{L}(v)$ holds by Lemma \ref{scli}, where
$L=N_I(u)\cap N_I(v)$. Set ${M_1}^*=M_1\setminus
\bigcup_{i=2:l}M_i$ and ${M_l}^*=M_l\setminus
\bigcup_{i=1:l-1}M_i$. To find the maximum independent set of ${\mathcal{I}^*}^c$, it suffices to compute the maximum independent sets of all $O(n)$ connected components of ${\mathcal{I}^*}^c$ and then taking the union of all the solutions. We claim that for any component of
${\mathcal{I}^*}^c$ say $H$, the induced sub-graph of $H$ on the
subset ${M_1}^*\cup {M_l}^*$, say $H_1$ is a complete bipartite
graph. It is not hard to see that $H_1$ is bipartite since all
vertices in ${M_1}^*$ as well as ${M_l}^*$ are adjacent in
${\mathcal{I}}^*$ and then non adjacent in ${\mathcal{I}^*}^c$.
Suppose that $A,B$ are bipartition of the vertices of $H_1$. For any
$u\in A$ and $v\in B$ we claim there is an edge between them in
${\mathcal{I}^*}^c$ and thus in $H_1$. Note that $H_1$ is connected
then let $u=u_1,u_2,\ldots,u_s=v$ be a $(u,v)$ path in $H_1$. For
$i=1,2,\ldots,s$, let $M_{l_i}$ be the maximal clique of $I$ which
contains $u_i$ so that $l_i$ is minimum between indices of
maximal cliques containing $u_i$. Note that $l_1$ must be $1$ and
$l_s$ must be $l$. Since there is an edge between $u_i$ and
$u_{i+1}$ and as $M_{l_i}\subseteq N_I(u_i)$, for each $i$,
therefore we have the following equalities.
$$\overrightarrow{N}_{M_{l_i}\cap M_{l_{i+1}}}(u_i)=\pm
\overrightarrow{N}_{M_{l_i}\cap M_{l_{i+1}}}(u_{i+1})\hspace{1cm}
i=1,...,l-1,\hspace{1cm}(*)$$ On the other hand Lemma \ref{or},
implies that $M_1\cap M_l\subseteq M_i$ for each $i=1,...,l$
hence $M_1\cap M_l\subseteq M_{l_i}\cap M_{l_{i+1}}$, for any
$i=1,...,s$. Thus the equality $(*)$ implies the following
equalities.
$$\overrightarrow{N}_{M_{1}\cap
M_{l}}(u_1)=\pm\overrightarrow{N}_{M_{1}\cap
M_{l}}(u_{2})=\ldots=\pm\overrightarrow{N}_{M_{1}\cap
M_{l}}(u_{s}).$$ Note that $u$ and $v$ are not adjacent in $I$. So, if they were adjacent in $\mathcal{I}^*$ they would belong to some $UC_4$, and the two other vertices of this $UC_4$ must be in $M_1\cap M_l$. The above equalities imply that $u,v$ must be
adjacent in $H_1$ hence $H_1$ is complete bipartite. This implies
that a maximum independent set in any component of
$\mathcal{I}*^c$ either have empty intersection with ${M_1}^*$ or
${M_l}^*$. Hence a maximum independent set of any of $O(n)$ components
of $\mathcal{I}*^c$ can be found by looking at one of the subgraphs
induced on $M_2\cup M_3\cup\cdots\cup M_l$ or $M_1\cup
M_2\cup\cdots\cup M_{l-1}$. Note that in the first step, we launch the method twice on two subgraphs with $l-1$ maximal cliques, the first step can be done in a polynomial time say $P_1(n)$, by the induction hypothesis. For each of the two subgraphs with $l-1$ maximal cliques, we launch again the method on subgraphs with $l-2$ maximal cliques in polynomial time, say $P_2(n)$, and so on. Then by production principle the whole procedure will take the time $\prod_{i=1}^{l}P_i(n)$. Since $l$ is bounded by a fixed number $l_0$ hence the independence number of ${\mathcal{I}^*}^c$ can be found in a polynomial time. Therefore the clique number of $\mathcal{I}^*$ and thus the S-clique number
of $\mathcal{I}$ can be found in polynomial time.}
\end{proof}

Now we consider the coloring problem of signed interval graphs. It is proved that this problem is NP-complete, it is in contrast with the existing results for the complexity of ordinary coloring problem in interval graphs.
\begin{thm}  The INTSCOL problem is NP-complete.
\end{thm}
\begin{proof} {We make use of Fact \ref{fact-3} to prove the assertion. We reduce
an instance of the MBC problem to an instance of the INTSCOL
problem. Let $(G,k)$ be an instance of the MBC problem, where $G$
is a graph with $n$ vertices. Consider the graph $I$ to be equal
to $K_n\wedge K_n$ which is the disjoint union of two copies of
the complete graph $K_n$, say $I_1$ and $I_2$. The graph $I$ is
clearly an interval graph. Let $\mathcal{I}=(I,\sigma)$ be a
signed graph in which the edges of the subgraph $I_1$ are all
positive and in $I_2$, the edges of a subgraph $G'$ isomorphic to
$G$ be assigned $-$ and $+$ elsewhere. We prove that the MBC
problem for the instance $(G,k)$ is correct if and only if the
INTSCOL problem for $(\mathcal{I},2n-k)$ is correct. On the other
hand $G$ has a biclique on more than $k$ vertices if and only if
$\mathcal{I}$ has a proper coloring with $2n-k$ colors. Suppose
that a switching mate of $\mathcal{I}$ say $\mathcal{I'}$ admits a
feasible coloring $c$ with $2n-k$ colors. Let $\mathcal{I}_1$,
$\mathcal{I}_2$ ($\mathcal{I'}_1$, $\mathcal{I'}_2$) be the
subgraphs of $\mathcal{I}$ ($\mathcal{I'}$) with ground  $I_1$
and $I_2$ respectively. With no loss of generality we may assume
the vertices of $\mathcal{I}'_1$ have been assigned the colors
$1,2,\ldots,n$ in the coloring $c$ and some colors from
$[n]=\{1,2,\ldots,n\}$, say $i_1,i_2,\ldots,i_r$ are used for
coloring some vertices of the subgraph $\mathcal{I}'_2$. Set $W$
be the set of such vertices in $\mathcal{I}'_2$ sharing the
colors with some of vertices in $\mathcal{I}'_1$. As it has been
used $2n-k$ colors in the coloring of $\mathcal{I'}$ then the
number of common colors between vertices of $\mathcal{I'}_1$ and
$\mathcal{I'}_2$ must be equal to $k$. On the other hand the
subgraph $\mathcal{I'}_1$ can be considered all positive hence
the subgraph of $\mathcal{I'}_2$ induced on $W$, must be all
positive as well. By Theorem \ref{har},
the subgraph of $\mathcal{I^-}$, induced on $W$ must be a
biclique which is of order $k$. Therefore the subgraph induced on
$W$ is a biclique of $G(G')$ on $k$ vertices. Thus the MBC problem
of instance $(G,k)$ is correct. The converse also holds as if $G$ has
a biclique on more than $k$ vertices, say $U$, the signed graph
$\mathcal{I}$ can be resigned to a signed graph $\mathcal{I}'$ in
which the subgraph induced on $U$ becomes all positive. Now one
can assign a proper coloring of $\mathcal{I}'$ with $2n-k$
colors. Such that all vertices in $\mathcal{I}'_1$ be colored
with $n$ distinct colors, say $1,2,...,n$, vertices of $U$ in
$\mathcal{I}'_2$ to be colored with $1,2,...,k$ and the rest of
vertices of $\mathcal{I}'_2$ be given the colors
$n+1,\ldots,2n-k.$  Hence the assertion follows by Fact
\ref{fact-3}. }
\end{proof}

\section{Miscellaneous}
In this section, we present two minor results. We consider the minimum signature problem abbreviated to MS which is defined below.

\noindent\textbf{MS} problem\\
Input: A signed graph $(G,\sigma_0)$ and a positive integer $k$\\
Question: Is there an edge set $\sigma$ where $|\sigma|\leq k$ and $(G,\sigma_0)$ is switching equivalent to $(G,\sigma)$?

\begin{prop}  The MS problem for signed graphs is NP-complete.
\end{prop}
\begin{proof} { The MS problem belongs to NP indeed. Let $(G,k)$ be an instance of the MEC problem. We reduce the problem to the instance $((G,E(G)),|E(G)|-k)$ of the MS problem. Consider the signed graph $(G,E(G))$ and let $(G,\sigma)$ be a switching equivalent signed graph of it where $\sigma$ is of size less than $|E(G)|-k$. Suppose that $(G,\sigma)$ is obtained from $(G,E(G))$ by resigning at the vertex set $S$. Resigning of the vertices in $S$ does not
change the sign of edges in the subgraphs $\langle S\rangle$
and $\langle \overline{S}\rangle$ of $G$. But it changes the sign
of all the edges in the edge cut $[S,\overline{S}]$. But $|\sigma|\leq |E(G)|-k$ if and only if $[S,\overline{S}]$ has size more than or equal to $k$. Thus the instance $(G,k)$ of the MEC problem is correct if and only if the instance $((G,E(G)),|E(G)|-k)$ of the MS problem is correct, hence the assertion follows by
Fact \ref{fact-4}.
}
\end{proof}

\begin{prop}  Let $\mathcal{I}=[I,\sigma]$ be a signed interval graph with connected ground. The following statements hold:
\begin{itemize}
\item[{\rm (a)}] $\chi_s(\mathcal{I})=2$ if and only if $I$ is a caterpillar tree that is a tree in which removing leaves produces a path.
\item[{\rm (b)}] If $\chi_s(\mathcal{I})=3$, then $I$ is an interval graph with clique number $3$ and
  all the triangles have the same signature.
\end{itemize}
\end{prop}

\begin{proof} {For the case {\rm (a)}, it is well known that the interval graphs do not have an induced cycle of length greater than three, therefore any connected signed interval graph with chromatic number two, say $I$ is a tree. It is proved in \cite{93} that an interval graph is triangle free if and only if it is a caterpillar tree. On the other hand all signed trees are switching equivalent therefore the assertion holds.

For case {\rm (b)}, the graph $I$ contains a cycle and then a triangle. The graph does not contain both balanced and unbalanced triangles since if there is a balanced triangle then all possible pairs of three colors are used in coloring the vertices of it therefore by definition the vertices of an unbalanced triangle can not be colored with the same three colors. On the other hand it is known that the chromatic number of an interval graph is equal to its clique number, thus the assertion follows.}
\end{proof}

\begin{center}
 {\bf Acknowledgments }
\end{center}
 We would like to thank the anonymous referees for their thorough reading the manuscript and valuable comments
 which considerably improved the contents of current paper. The research is in part supported by a grant from the ANR project HOSIGRA (ANR-17-CE40-0022).


\begin{thebibliography}{MM}

\bibitem{BFHN}
{\sc R. C. Brewster, F. Foucaud, P. Hell and R. Naserasr,} The complexity of signed graph and edge-coloured graph homomorphisms, \textit{Discrete Mathematics} \textbf{340} (2017) 223--235.

\bibitem{93}
{\sc J. Eckhoff,} { Extremal interval graphs} {\em Journal of Graph Theory } \textbf{17} (1993) 117--127.

\bibitem{FI}
{\sc P. C. Fishburn}, { Interval orders and interval
graphs: A study of partially ordered sets}, Wiley-Interscience
Series in Discrete Mathematics, New York: John Wiley \& Sons
(1985).


\bibitem{HA}
{\sc F. Harary}, {On the notion of balance of a signed graph}, {\em The Michigan Math. Journal}, \textbf{2} (1954) 143-�146.

\bibitem{PEE}

{\sc J. M. Lewis and M. Yannakakis}, The node-deletion problem for hereditary properties is NP-complete. \textit{Journal of Computer and System Sciences} \textbf{20}(2), (1980) 219--230.

\bibitem{LO}
{\sc P. J. Looges} and {\sc S. Olariu},  Optimal greedy
algorithms for indifference graphs, {\em Computers \& Mathematics with
Applications}, \textbf{25} (1993) 15--25.


\bibitem{MRS}
{\sc    E. Macajova,  A. Raspaud and  M. Skoviera}, The Chromatic Number of a Signed Graph, {\em The Electronic Journal of Combinatorics} \textbf{23} (2016) 1--14.

\bibitem{NAS1}
{\sc R. Naserasr}, {\sc E. Rollova} and {\sc E. Sopena},  Homomorphisms of signed graphs, {\it J. Graph Theory},
\textbf{79} (2015) 178--212.

\bibitem{SCH}
{\sc  A. Schrijver,} {\it Combinatorial Optimization, Polyhedra and Efficiency}.

\bibitem{SHI}
{\sc D. R. Shier,} Some aspects of perfect elimination orderings in chordal graphs, \textit{Discrete Appl. Math.,}
\textbf{7} (1984), 325--331.

\bibitem{ZAS}
{\sc T. Zaslavsky},  Signed graphs, \textit{Discrete Appl. Math.,} \textbf{4} (1982), 47--74.

\bibitem{ZAS0}
{\sc T.  Zaslavsky},   Chromatic  invariants  of  signed  graphs.
\textit{ Discrete Mathematics} \textbf{42} (1982) 287-�312.

\bibitem{ZAS1}
{\sc T. Zaslavsky}, Signed graph coloring, \textit{Discrete Mathematics} \textbf{39}, (1982) 215--228.




\end{thebibliography}
\end{document}